\documentclass[11pt]{amsart}

\usepackage{geometry}               
\usepackage{amsmath, amssymb, amsthm}
\usepackage{xcolor}
\definecolor{myurlcolor}{rgb}{0,0,0.4}
\definecolor{mycitecolor}{rgb}{0,0.5,0}
\definecolor{myrefcolor}{rgb}{0.5,0,0}
\usepackage[pagebackref,colorlinks,linkcolor=myrefcolor,citecolor=mycitecolor,urlcolor=myurlcolor,hypertexnames=true]{hyperref}
\usepackage{amsrefs}
\usepackage{graphicx}
\usepackage{pgfplots}
\usepackage{enumitem}
\usepackage{microtype}

\theoremstyle{plain}
\newtheorem{theorem}{Theorem}[section]

\newtheorem{lemma}[theorem]{Lemma}

\newtheorem{corollary}[theorem]{Corollary}
\newtheorem{proposition}[theorem]{Proposition}

\theoremstyle{remark}

\newtheorem{remark}[theorem]{Remark}

\theoremstyle{definition}
\newtheorem{definition}[theorem]{Definition}
\newtheorem{example}[theorem]{Example}
\newtheorem{problem}[theorem]{Problem}

\input{xy}
\xyoption{all}

\def\im{{\rm im}}

\def\C{\mathbb{C}}
\def\R{\mathbb{R}}

\def\id{\mathrm{id}}
\def\N{\mathbb{N}}

\let\originalleft\left 
\let\originalright\right 
\renewcommand{\left}{\mathopen{}\mathclose\bgroup\originalleft} 
\renewcommand{\right}{\aftergroup\egroup\originalright} 

\title[Spectrahedral Containment and Operator Systems]{Spectrahedral Containment and Operator Systems with finite-dimensional Realization}

\author{Tobias Fritz}
\address{Tobias Fritz, Max Planck Institute for Mathematics in the Sciences, Leipzig, Germany }
\email{fritz@mis.mpg.de}

\author{Tim Netzer}
\address{Tim Netzer, Universit\"at Innsbruck, Austria}
\email{tim.netzer@uibk.ac.at}

\author{Andreas Thom}
\address{Andreas Thom, TU Dresden, Germany}
\email{andreas.thom@tu-dresden.de}

\begin{document}

\begin{abstract} 
Containment problems for polytopes and spectrahedra appear in various applications, such as linear and semidefinite programming, combinatorics, convexity and stability analysis of differential equations. This paper explores the theoretical background of a method proposed by Ben-Tal and Nemirovksi \cite{bental}. Their method provides a strengthening of the containment problem, that is algorithmically well tractable.  
 To analyze this method, we study abstract operator systems, and investigate when they have  a finite-dimensional concrete realization. 
Our results give some profound insight into their approach.  They  imply that when testing the inclusion of a fixed polyhedral cone in an arbitrary spectrahedron, the strengthening  is tight if and only if the polyhedral cone is a simplex. This is true independent of the representation of the polytope. We also deduce error bounds in the other cases, simplifying and extending recent results by various authors.

\end{abstract}

\maketitle

\section{Introduction and Preliminaries}
Spectrahedra are the feasible sets of semidefinite programming, and have attracted a lot of attention in recent years, both from an applied and pure perspective (see \cite{blek} for an overview). Studying their geometry is a rewarding task for pure mathematicians, but any insight also directly influences the numerous applications, as are optimization, convexity, control theory and others. One of these applications, Lyapunov stability analysis of differential equations, was studied by Ben-Tal and Nemirovski \cite{bental}. The problem reduces to checking containment of a box in a spectrahedron, which is a hard problem in general. They came up with a strengthening of this problem, which admits an efficient algorithmic  approach. 
 It has been discovered  \cite{hecp} that the method can only be fully understood by adding matricial levels to the spectrahedra, i.e.~by examining their free versions. This idea has been further pursued in \cite{dav2, hedi, ke1,ke2,ke3}.

The results in this paper can be looked at from two sides. On the one hand, we analyze the power of Ben-Tal and Nemirovski's idea, providing a complete description of the cases in which their method provides tight results. In the case of non-tightness  we provide error bounds, simplifying and extending 
upon recent results of several authors. Since tightness is a rather rare phenomenon, such error bounds are of particular interest for applications.  Our results show how they directly emerge from geometric properties of the problem, and that they can be computed explicitly.  On the other hand, we examine abstract operator systems, and ask when these admit a finite-dimensional concrete realization, i.e.~a realization by matrices. This is an interesting and hard problem, which often involves determining the boundary representations of the system (see for example \cite{arg, arv, dav}).  The connection between the two perspectives becomes clear by observing that free spectrahedra are essentially the same as operator systems with finite-dimensional realizations. We believe that only the fusion of these two views allows to fully understand the setup of Ben-Tal and Nemirovski's seminal approach, and further exploit the capabilities of spectrahedra in applications.

Our paper is structured as follows. We start with an abstract operator system and characterize when it admits a finite-dimensional realization (Theorem \ref{finrea}). We then investigate operator systems constructed from convex cones at scalar level, namely the smallest and the largest operator system of a cone. We show that the largest system admits a finite-dimensional realization if and only if the cone is polyhedral (Theorem \ref{maxreal}), and the  smallest system of a polyhedral cone is finite-dimensional realizable if and only if the cone is a simplex (Theorem \ref{simplex}). The smallest system of a non-polyhedral cone can also be finite-dimensional realizable (Example~\ref{circle}), but this seems to happen very rarely. Now translated into the initial problem of testing inclusion of spectrahedra, Theorem~\ref{simplex} says the following. When checking inclusion of a fixed polytope in an arbitrary spectrahedron, the strengthening first introduced by Ben-Tal and Nemirovski \cite{bental} is tight for any spectrahedron if and only if the polytope is a simplex (Corollary \ref{incmain}). This is true independently of the representation of the polytope. Further, our setup allows to give an easy proof of the existence of scaling factors for inclusion from  \cite{dav, ke3}, and prove novel bounds for general spectrahedra (see Section \ref{sec_inc}). 

Let us introduce the basic concepts. Throughout, $\mathcal{V}$ denotes a $\C$-vector space with involution $*$, and $\mathcal{V}_h$ is the $\R$-subspace of Hermitian elements. For any $s\geq 1$, the vector space $\mathbb M_s(\mathcal{V}) = \mathcal{V}\otimes_{\C}\mathbb M_s(\C)$ of $s\times s$-matrices with entries from $V$ comes equipped with the canonical involution defined by $\left(v_{ij}\right)_{i,j}^*:=\left(v_{ji}^*\right)_{i,j}$.

\begin{definition}[{e.g.~\cite[Chapter~13]{pau}}]
An {\it abstract operator system} $\mathcal{C}$ on $\mathcal{V}$ consists of a closed and salient convex cone $\mathcal{C}_s\subseteq \mathbb M_s(\mathcal{V})_h$ for each $s\geq 1$, such that
\begin{itemize}
\item $A\in \mathcal{C}_s, V\in\mathbb M_{s,t}(\mathbb C) \Rightarrow V^*AV\in \mathcal{C}_t$,
\item there is $u\in \mathcal{C}_1\subseteq \mathcal{V}_h$ such that $u\otimes I_s$ is an order unit (or equivalently interior point) of $\mathcal{C}_s$ for all $s\geq 1$.
\end{itemize}
\end{definition}

\renewcommand{\labelenumi}{(\alph{enumi})}
\renewcommand{\theenumi}{(\alph{enumi})}

\begin{remark}
\begin{enumerate}[leftmargin=*]
\item The topology in which each $\mathcal{C}_s$ is required to be closed is understood to be the finest locally convex topology on $\mathcal V$.
\item We usually consider the order unit $u\in \mathcal{C}_1$ to be part of the structure of an operator system (as opposed to a mere property), which means that maps of operator systems are typically required to preserve it.
\item $(u\otimes I_s)\in \mathcal{C}_s$ is an order unit for all $s$ if and only if this holds for $s=1$. To show this, we start with an arbitrary element $A\in \mathbb{M}_s(\mathcal V)_h$ and decompose it as $A=\sum_{i=1}^n v^{(i)} \otimes M_i$ with $v^{(i)}\in \mathcal V_h$ and $M_i\in\mathbb{M}_s(\C)_h$. Assuming that $u\in \mathcal{C}_1$ is an order unit, choose $\lambda\in\R$ such that $\pm v^{(i)} + \lambda u \in \mathcal{C}_1$ for all $i$, and write $M_i = P_i - Q_i$ as a difference of two positive semidefinite matrices. Then
\[
	\sum_i (v^{(i)}+\lambda u)\otimes P_i + (-v^{(i)}+\lambda u)\otimes Q_i= \sum_i v^{(i)} \otimes M_i + \lambda u\otimes \sum_i (P_i+Q_i) 
\]
is also in $\mathcal{C}_s$. Thus if $\gamma\geq 0$ is large enough to ensure $\gamma I_s -\sum_i (P_i+Q_i)\geqslant 0$, then $A+\gamma\lambda \left(u\otimes I_s\right)\in \mathcal{C}_s$. So $u\otimes I_s$ is indeed an order unit for $\mathcal{C}_s$.
\end{enumerate}
\end{remark}

By the Choi--Effros Theorem (\cite{choi}, see  also \cite[Chapter~13]{pau}), for any abstract operator system $\mathcal{C}$  there is a Hilbert space $\mathcal H$ and a $*$-linear mapping $\varphi\colon \mathcal V\rightarrow \mathbb B(\mathcal H)$ with $\varphi(u)={\rm id}_{\mathcal H}$, such that for all $s\geq 1$ and $A\in \mathcal{C}_s$,
$$
A\in \mathcal{C}_s \:\Leftrightarrow\: (\varphi\otimes{\rm id})(A)\geqslant 0.
$$
On the right-hand side, we use the canonical identification
$$
\mathbb M_s(\mathbb B(\mathcal H))=\mathbb B(\mathcal H)\otimes_\C \mathbb M_s(\C) = \mathbb B(\mathcal H^s)
$$
to define positivity of the operator. Such a mapping $\varphi$ is called a {\it concrete realization} or just {\it realization} of the operator system $\mathcal{C}$. A realization $\varphi$ is necessarily injective, since $\mathcal{C}_1$ does not contain a nontrivial subspace.

\begin{definition}
For $r\in\N$, an abstract operator system $\mathcal{C}$ is {\it $r$-dimensional realizable} if there is a realization with $\dim{\mathcal{H}}=r$. It is {\it finite-dimensional realizable} if it is $r$-dimensional realizable for some $r\in\N$.
\end{definition}

Now assume that $\mathcal V$ is finite-dimensional. After a suitable choice of basis, we can assume $\mathcal V=\mathbb C^d$ with the canonical involution, and thus  $\mathcal V_h=\R^d$. Then
$$
\mathbb M_s(\mathcal V)=\mathcal V\otimes_\C \mathbb M_s(\C)=\mathbb M_s(\C)^d,\quad \mathbb M_s(\mathcal V)_h={\rm Her}_s(\C)^d,
$$
and a realization of $\mathcal{C}$ just consists of self-adjoint operators $T_1,\ldots, T_d\in\mathbb B(\mathcal H)_h$ with $u_1T_1+\cdots+u_dT_d={\rm id}_{\mathcal H}$ and 
$$
(A_1,\ldots, A_d)\in \mathcal{C}_s \:\Leftrightarrow\: T_1\otimes A_1+\cdots + T_d\otimes A_d\geqslant 0.
$$
Finite-dimensional realizability then means that the $T_i$ can be taken to be matrices.

\begin{definition}
A {\it (classical) spectrahedral cone} is a set of the form
$$
\left\{ a\in\R^d \bigm|  a_1M_1+\cdots+a_dM_d\geqslant 0\right\},
$$
where $M_1,\ldots, M_d\in {\rm Her}_r(\C)$ are Hermitian matrices, and $\geqslant 0$ again denotes positive semidefiniteness. For any $s\geq 1$, we define
$$
\mathcal S_s(M_1,\ldots, M_d) := \left\{ (A_1,\ldots, A_d)\in{\rm Her}_s(\C)^d \bigm| M_1\otimes A_1+\cdots +M_d\otimes A_d\geqslant 0\right\}.
$$
The family of cones $\mathcal S(M_1,\ldots,M_d)=\left(\mathcal S_s(M_1,\ldots, M_d)\right)_{s\geq 1}$ is called the {\it free spectrahedron} defined by $M_1,\ldots,M_d$.
\end{definition}

\begin{remark}
\label{FSvsOS}
In order for a free spectrahedron to be an operator system, the positive cones must be salient and have an order unit. The first is equivalent to the $M_i$ being linearly independent, and the latter happens in particular if there is $u\in\R^d$ with $\sum_i u_i M_i = I_r$, in which case we take this $u$ to be the order unit.
\end{remark}

Classical spectrahedra are the feasible sets of semidefinite programming, which allows for efficient numerical algorithms (see for example \cite{sedumi, wo}). They share many properties of polytopes, which form a strict subclass. It is generally hard to decide whether a cone is spectrahedral, and a lot of recent research deals with questions arising in this area (see \cite{blek} for an overview). For example, the {\it inclusion problem} in its basic form asks whether $$\mathcal S_1(M_1,\ldots, M_d)\subseteq\mathcal S_1(N_1,\ldots, N_d)$$ holds for given families of matrices $M_i$ and $N_j$. In Section \ref{sec_inc} we will explain how this problem relates to our results.  For the moment, just note that a free spectrahedron with the properties of Remark~\ref{FSvsOS} is (up to isomorphism) the same as a finite-dimensional realizable operator system.

\section{A criterion for finite-dimensional realizations}

In this section, we prove a criterion for operator systems to admit a finite-dimensional realization, namely Theorem \ref{finrea} below. Throughout, let $\mathcal{C}=\left(\mathcal{C}_s\right)_{s\geq 1}$ be an operator system on $\mathcal{V}=\C^d$ with order unit $u=(u_1,\ldots,u_d)\in\mathcal{C}_1$. Let $\mathcal{C}_s^{\vee}$ denote the dual cone of $\mathcal{C}_s$, i.e.~the set of all $*$-linear functionals on $\mathbb M_s(\C)^d$ that are nonnegative on $\mathcal{C}_s$. We begin by reviewing the separation method of Effros and Winkler. 

\begin{lemma}[\cite{effwi}]
Let $\varphi\in\mathcal{C}_r^\vee$ be such that $\varphi(u\otimes vv^*) > 0$ for all $0\neq v\in\C^r$. Then there are $M_1,\ldots,M_d\in {\rm Her}_r(\C)$ with $\sum_i u_i M_i = I_r$, which generate a free spectrahedron containing $\mathcal{C}$, and such that:
\begin{enumerate}
	\item\label{sepboundary} If $A\in{\rm Her}_s(\C)^d$ is such that $\varphi(V^* A V) = 0$ for some $V\neq 0$, then $A$ is in the boundary of this free spectrahedron.
	\item\label{sepoutside} If $A\in{\rm Her}_r(\C)^d$ is such that $\varphi(A) < 0$, then $A$ is not in this free spectrahedron.
\end{enumerate}
\label{effwisep}
\end{lemma}

\begin{proof}
Let the $N_1,\ldots,N_d\in {\rm Her}_r(\mathbb C)$ be such that $\varphi(B_1,\ldots, B_d)=\sum_i {\rm tr}\left(\overline{N}_iB_i\right)$ for all $B\in\mathbb M_r(\mathbb C)^d$. The positivity assumption guarantees that $\sum_i u_i N_i \geq 0$. Even better, the assumption $\varphi(u\otimes vv^*) > 0$ for all $0\neq v\in\C^r$ implies that $\hat{N} := \sum_i u_i N_i > 0$, and thus we can put $M_i := \hat{N}^{-1/2} N_i \hat{N}^{-1/2}$ and have $\sum_i u_i M_i = I_r$ by construction.
	
To show that the resulting free spectrahedron contains $\mathcal{C}$, consider $A\in\mathcal{C}_s$. Then for $x=\sum_{j=1}^r e_j\otimes v_j$ with $v_1,\ldots, v_r\in\mathbb C^s$ and $e_1,\ldots, e_r$ the standard basis of $\mathbb C^r$, we have
\begin{equation}
\label{sepcalc}
\left\langle x,\left(\sum_i N_i\otimes A_i\right)x \right\rangle= \sum_i {\rm tr}\left(\overline{N}_i V^*A_iV\right)=\varphi(V^*AV)\geq 0,
\end{equation}
where $V$ is the matrix with $v_1,\ldots, v_r$ as its columns. Therefore $\sum_i N_i\otimes A_i\geqslant 0$, which also implies $\sum_i M_i\otimes A_i\geqslant 0$, as was to be shown.
	
If $\varphi(V^*AV)=0$ for some $V\neq 0$, then $A$ lies in the boundary of the free spectrahedron, since~\eqref{sepcalc} shows that $\sum_i N_i\otimes A_i$ and hence also $\sum_i M_i \otimes A_i$ is not positive definite, resulting in~\ref{sepboundary}. Part~\ref{sepoutside} works similarly.
\end{proof}

\begin{definition}
The  {\it essential boundary} of $\mathcal C$ is:
$$
\partial^{\rm ess} \mathcal{C}_s:=\left\{ A\in \mathcal{C}_s\mid \exists \varphi\in \mathcal{C}_s^{\vee}, \: \varphi(u\otimes vv^*)>0 \mbox{ for all }  v\in\mathbb C^s\setminus\{0\},\: \varphi(A)=0 \right\}.
$$
So an element is in the essential boundary if and only if its minimal exposed face does not contain an element $u\otimes vv^*$ with $v\neq 0$.
\end{definition}

Example~\ref{essbound} showcases what the essential boundary of a particular operator system may look like.

\begin{theorem}\label{finrea}
A finite-dimensional operator system $\mathcal{C}$ is $r$-dimensional realizable if and only if it has the following property: for any $n, s_1,\ldots, s_n\in\N$ and $A^{(i)}\in\partial \mathcal{C}_{s_i}$, there exist $0\neq V_i\in\mathbb M_{s_i,r}(\mathbb C)$ with 
\[
\sum_{i=1}^n V_i^* A^{(i)} V_i \:\in\: \partial^{\rm ess} \mathcal{C}_r.
\]
\end{theorem}

\begin{proof}
First assume that the system is $r$-dimensional realizable, with defining matrices $M_1,\ldots, M_d\in {\rm Her}_r(\mathbb C)$. For $A^{(i)}=(A_1^{(i)},\ldots, A_d^{(i)})\in \partial \mathcal{C}_{s_i}$, there exist vectors $v_1^{(i)},\ldots, v_r^{(i)}\in\mathbb C^{s_i}$  such that
$$
x^{(i)}:=\sum_{k=1}^r e_k\otimes v_k^{(i)}\neq 0,
$$
 and
$$
\left(\sum_{j=1}^d M_j\otimes A_j^{(i)}\right) x^{(i)}=0.
$$
Let $V_i$ be the matrix with columns $v_1^{(i)},\ldots, v_r^{(i)}$. Then $V_i\neq 0$, and some calculation analogous to~\eqref{sepcalc} shows that
\[
	{\rm tr}\left(\sum_{j=1}^d \overline{M}_j \sum_{i=1}^n V_i^* A_j^{(i)} V_i\right) = \sum_{i=1}^n\left\langle x^{(i)},\left(\sum_{j=1}^d M_j\otimes A_j^{(i)}\right) x^{(i)}\right\rangle =0.
\]
This proves that $\sum_i V_i^* A^{(i)} V_i\in\partial^{\rm ess} \mathcal{C}_r$, since the positive functional $B\mapsto {\rm tr}(\sum_j \overline{M}_j B_j)$ is strictly positive on each $u\otimes vv^*$ with $v\neq 0$.

For the converse direction, we use one of the key arguments from \cite{freelmi}. Let $A^{(i)}\in\partial\mathcal{C}_{s_i}$ for $i=1,\ldots,n$ be elements of the boundary. Then the assumption guarantees that there are $V_i\neq 0$ with $\sum_i V_i^* A^{(i)} V_i\in \partial^{\rm ess}\mathcal{C}_r$. This means that there is $\varphi\in\mathcal{C}^{\vee}_r$ with $\varphi(u\otimes vv^*) > 0$ for all $v\in\C^s\setminus\{0\}$ and $\varphi\left(\sum_i V_i^* A^{(i)} V_i\right) = 0$. Since $\varphi\in\mathcal{C}_r^\vee$, this implies that $\varphi\left(V_i^* A^{(i)} V_i\right) = 0$ for each $i$ separately. Hence Lemma~\ref{effwisep} constructs matrices in ${\rm Her}_r(\C)$ which generate a free spectrahedron containing $\mathcal{C}$, and such that the $A^{(i)}$ are in its boundary.

The existence of the order unit implies that the defining matrices of such a free spectrahedron are uniformly bounded. Therefore the tuples of matrices that define free spectrahedra containing $\mathcal{C}$ and satisfy $\sum_i u_i M_i = I_r$ form a compact set in $\mathrm{Her}_r(\mathbb{C})^d$. We now choose a sequence of boundary elements $A^{(i)}\in\mathcal{C}_{s_i}$ that are dense in the boundary at all matrix levels, and consider the sequence of free spectrahedra associated to all finite initial subsequences. By compactness, this sequence of free spectrahedra containing $\mathcal{C}$ must have an accumulation point. The free spectrahedron described by such an accumulation point again contains $\mathcal{C}$, and every $A^{(i)}$ is in its boundary. We therefore have an $r$-dimensional realizable system which has the same boundary as $\mathcal{C}$, and thus coincides with $\mathcal{C}$.
\end{proof}

We will see in Section \ref{sec_min} how this result can be used to show that certain operator systems are not finite-dimensional realizable. 

\section{The Largest Operator System of a Cone}\label{sec_max}

In this and the next section, we start with a closed salient cone $C\subseteq \R^d$ with order unit $u$ and consider operator systems $\left(  C_s\right)_{s\geq 1}$ with $C_1=C$. It is not hard to see that there is always a smallest and a largest one, as has also been noticed in \cite{pauto}\footnote{Let us emphasize that the largest operator systems in our paper are called \emph{minimal} in \cite{pauto}, while our smallest ones are called the \emph{maximal} ones of~\cite{pauto}. The reason is that the norm induced by a set-theoretically larger system is smaller, and vice versa. So if one is interested in the comparison with operator spaces, then it makes sense to adopt the conventions of~\cite{pauto}. We decided to stick with the set-theoretic notions, hoping that this is less confusing to our readers.}. We start with the largest system:
$$
C_s^{\max}:=\left\{ (A_1,\ldots, A_d)\in{\rm Her}_s(\mathbb C)^d \bigm| \forall v\in \mathbb C^s \ (v^*A_1v,\ldots, v^*A_dv)\in C\right\}.
$$
We also write $ C^{\max}$ as shorthand for the family $\left(  C_s^{\max}\right)_{s\geq 1}$. This system is largest in the sense that for any operator system $\left(  D_s\right)_{s\geq 1}$ with $ D_1\subseteq C$, we have $ D_s\subseteq  C_s^{\max}$ for all $s$. 

The following proposition is a technical ingredient for the main result of this section, Theorem~\ref{maxreal}.

\begin{proposition}\label{comei}
For $M,N\in{\rm Her}_s(\mathbb C)$, define
$$
\lambda_1:=\min \left\{ \lambda \in\mathbb R \biggm| \left(\begin{array}{cc}M+\lambda I & N \\N &  I\end{array}\right)\geqslant 0\right\}
$$
and 
$$
\lambda_2 := \min \left\{\lambda\in\mathbb R\mid \vert w_1\vert ^2M  +2\,{\rm Re}(w_1\overline{w}_2)N  + (\lambda \vert w_1\vert^2+\vert w_2\vert^2)I\geqslant 0\quad\forall w\in\mathbb C^2\right\}.
$$
Then $\lambda_2\leq \lambda_1$, and if $\lambda_2=\lambda_1$, then $M$ and $N$ have a common eigenvector.
\end{proposition}

\begin{proof}
It is well-known that $\lambda_1=\max_{\Vert v\Vert =1} (v^*N^2v-v^*Mv)$. Concerning $\lambda_2$, it is easy to see that the inequality
$$
\vert w_1\vert^2M  + 2\,{\rm Re}(w_1\overline{w}_2)N  + (\lambda \vert w_1\vert^2+\vert w_2\vert^2)I\geqslant 0\quad  \forall w\in\mathbb C^2
$$
is equivalent to
$$
M+ 2rN+ (\lambda+r^2)I\geqslant 0 \quad \forall r\in\mathbb R,
$$
and thus to
$$
\left(v^*Nv\right)^2\leq v^*Mv +\lambda \quad\forall\Vert v\Vert =1.
$$
Therefore
$$
\lambda_2=\max_{\Vert v\Vert=1} [\left(v^*Nv\right)^2-v^*Mv].
$$
We know that $\left(v^*Nv\right)^2=(Nv)^*vv^*(Nv)\leqslant (Nv)^*I(Nv)= v^*N^2v$ for all $\Vert v\Vert=1$,  and thus $\lambda_2\leq \lambda_1$. Whenever $\left(v^*Nv\right)^2= v^*N^2v$, then $Nv\in {\rm ker}(I-vv^*)$, so $v$ is an eigenvector of $N$. Thus if $\lambda_2=\lambda_1$, then any $v$ that attains $\lambda_2$ must also attain $\lambda_1$, and therefore be an eigenvector of $N$.

We finally show that if $\lambda_2=\max_{\Vert v\Vert=1}[\left(v^*Nv\right)^2-v^*Mv]$ is attained at some eigenvector $v$ of $N$, then $v$ is also an eigenvector of $M$. We assume $\Vert v\Vert=1$ and choose an arbitrary $w$ with $\Vert w\Vert =1$ and $w\perp v$. Consider the smooth function $f:\R\to\R$ defined by
$$
f(\epsilon) := \left(\frac{(v+\epsilon w)^*}{\Vert v+\epsilon w\Vert}N\frac{(v+\epsilon w)}{\Vert v+\epsilon w\Vert}\right)^2- \frac{(v+\epsilon w)^*}{\Vert v+\epsilon w\Vert}M\frac{(v+\epsilon w)}{\Vert v+\epsilon w\Vert}
$$ and compute 
\[
f'(0)=-w^*Mv-v^*Mw,
\]
where the derivative of the first term vanishes since $v$ is an eigenvector of $N$. Since $v$ attains $\lambda_2$, there is a maximum of $f$ at $\epsilon=0$, and therefore $w^*Mv+v^*Mw=0$. This means ${\rm Re}(v^*Mw)=0$, and by using $-iw$ in place of $w$ also  ${\rm Im}(v^*Mw)=0$. Hence $v^*Mw=0$ for all $w$ with $w\perp v$, which means that the orthogonal complement of $v$ is invariant under $M$. But then $\mathbb Cv$ must also be invariant under $M$, so that $v$ is an eigenvector of $M$.
\end{proof}

We can now prove our main result on largest operator systems:

\begin{theorem}\label{maxreal}
The operator system $C^{\max}$ admits a finite-dimensional realization if and only if $C$ is polyhedral.
\end{theorem}

\begin{proof}
One direction is clear: if $C=\left\{ a\in \R^d\mid \ell_1(a)\geq 0, \ldots, \ell_r(a)\geq 0\right\}$, with linear functionals $\ell_i:\R^d\to\R$ such that $\ell_i(u)=1$ for all $i$, then for all $s\geq 1$,
$$
C_s^{\max}=\left\{ A\in{\rm Her}_s(\C)^d \bigm| (\ell_1\otimes\id)(A)\geqslant 0,\ldots, (\ell_r\otimes\id)(A)\geqslant 0\right\},
$$
and this gives rise to an $r$-dimensional realization with diagonal matrices. 

We now show that the largest system of a non-polyhedral cone does not admit a finite-dimensional realization. First, we argue that we can restrict to the case $d=3$. Indeed, every  non-polyhedral cone $C$ admits a $3$-dimensional linear section through $0$ and the order unit $u$, which is not polyhedral either~\cite[Theorem~4.7]{klee}, and a possible finite-dimensional realization of $ C^{\max}$ would restrict to a finite-dimensional realization of the largest system over this $3$-dimensional intersection-cone. So the case $d=3$ is enough to deal with. Moreover, we can assume that $C$ itself is spectrahedral, since otherwise there is not even a finite-dimensional realization of any system that coincides with $C$ at scalar level.

Now if $C\subseteq\R^3$ is non-polyhedral but spectrahedral, then there is an isomorphism $\varphi\in {\rm GL}_3(\R)$ such that $C\cap\varphi(C)$ has nonempty interior, but does not have a face of dimension $2$. Indeed, the Zariski closure of the boundary of $C$ is an algebraic variety, and hence there must be a smooth point with strict curvature by non-polyhedrality. A  reflection $\varphi$ at a suitable hyperplane close to such a point will then work. Since $ C_s^{\max}\cap  D_s^{\max}=(C\cap D)_s^{\max}$ holds for any two cones $C$ and $D$, and the intersection of two systems with finite-dimensional realization has a finite-dimensional realization, we can thus assume that $C$ does not have a face of dimension $2$.
 
Now assume  $M_1,M_2, M_3\in{\rm Her}_r(\mathbb C)$ are defining matrices for $C^{\max}$ of minimal matrix size $r$. For any $A=(A_1,A_2,A_3)\in{\rm Her}_s(\mathbb C)^3$, we then have
\begin{align*}
\sum_i M_i\otimes A_i \geqslant 0 & \:\Leftrightarrow\: A\in C_s^{\max} \\ 
& \Leftrightarrow\: v^*Av\in C=C_1^{\max} \quad \forall v\in\mathbb C^s\\
& \Leftrightarrow\: \sum_i M_i\cdot v^*A_iv\geqslant 0\quad \forall v\in \mathbb C^s \\
& \Leftrightarrow\: \sum_i w^*M_iw\cdot v^*A_iv \geq 0 \quad \forall v\in\mathbb C^s, w\in\mathbb C^r\\
& \Leftrightarrow\: \left\langle\left(\sum_i M_i\otimes A_i\right) x,x\right\rangle\geq 0\ \forall \mbox{ elementary tensors } x\in\mathbb C^r\otimes \mathbb C^s.
\end{align*}
Via a suitable change of basis in $\C^3$, we can arrive at $M_3 = I_r$. The above equivalence then entails that the matrix
\[
\left(\begin{array}{cc}M_1+\lambda I_r & M_2 \\ M_2 &  I_r\end{array}\right)
\]
is positive if and only if it is positive on all vectors of the form $\left(\begin{array}{c} w_1 v \\ w_2 v\end{array}\right)$ for $v\in\C^r$ and $w=(w_1,w_2)\in\C^2$. Using Proposition~\ref{comei}, it follows that $M_1$ and $M_2$, and trivially also $M_3$, have a common eigenvector. Thus we can split off a $1\times 1$-block in each $M_i$. Since the corresponding linear inequality is not needed in the linear inequalities description of $C$ (because there is no face of dimension $2$), it is also redundant in the description of the largest system. This contradicts the minimality of $r$.
\end{proof}

\section{The Smallest Operator System of a Cone}\label{sec_min}

Again let $C\subseteq \mathbb R^d$ be a closed salient convex cone with order unit $u$. Define
$$
C_s^{\min}:=\left\{\sum_i c_i\otimes P_i \biggm| c_i\in C, P\in{\rm Her}_s(\mathbb C), P\geqslant 0\right\}.
$$ 

\begin{lemma}
$C^{\min}$ is the smallest operator system with $C_1^{\min}=C$.
\end{lemma}

\begin{proof}
It is clear that $C^{\min}$ is contained in any operator system extending $C$. 

It remains to check that each $C_s^{\min}$ is closed. By Caratheodory's theorem, the number of elementary tensors required to reach every $A = \sum_i c_i\otimes P_i$ is uniformly bounded. Hence it is enough to show that the set of elementary tensors $\{c\otimes P\: :\: c\in C, P\in{\rm Her}_s(\mathbb{C})_+\}$ is closed. By choosing any tensor norm, it follows that the elementary tensors of norm $1$ are tensor products of elements of norm $1$ and therefore form a compact set.
\end{proof}

Since it will be a crucial ingredient in our main result, we compute the essential boundary of a particular smallest system:

\begin{example}\label{essbound}
Consider the cone over the square, i.e.~$C={\rm cc}\{v_1,v_2,v_3,v_4\}\subseteq\mathbb R^3$, where
\begin{equation}\label{squarevertices}
v_1=(1,-1,1),\quad v_2=(-1,1,1),\quad v_3=(1,1,1),\quad v_4=(-1,-1,1),
\end{equation}
and $u=(0,0,1)$.  For $A_1,A_2,A_3,A_4\geqslant 0$, we have
$$
v_1\otimes A_1+v_2\otimes A_2+v_3\otimes A_3+v_4\otimes A_4\in\partial^{\rm ess}C_s^{\min}
$$
if and only there is some $U\in{\rm GL}_s(\mathbb C)$ with
\begin{equation}\label{orthoims}
\im(UA_1)\perp {\rm im}(UA_2) \:\text{ and }\: \im(UA_3)\perp{\rm im}(UA_4).
\end{equation}
In fact, assume $\varphi\colon{\rm Her}_s(\mathbb C)^3\rightarrow \mathbb R$ is nonnegative on $C_s^{\min}$. Then
$$
\varphi(X)={\rm tr}\left(X_1 M_1+X_2 M_2+X_3 M_3\right)
$$
for some $M\in {\rm Her}_s(\mathbb C)^3$ with
$$
\pm M_1\pm M_2+M_3\geqslant 0
$$
for all four sign combinations. Furthermore,  $\varphi(u\otimes vv^*)>0$ for all $0\neq v\in\mathbb C^s$ just means that $M_3>0$. So there is some $U\in{\rm GL}_s(\mathbb C)$ with $(U^{-1})^*M_3 U^{-1} = I_s$. Now assume
\begin{align*}
0 & = \varphi\left(v_1\otimes A_1+v_2\otimes A_2+v_3\otimes A_3+v_4\otimes A_4\right)\\
& = \mathrm{tr}\left( A_1(M_1-M_2+M_3) +A_2(-M_1+M_2+M_3) \right.\\
& \left.\quad +A_3(M_1+M_2+M_3)+A_4(-M_1-M_2+M_3)\right).
\end{align*}
With $S=M_1+M_2$ and $D=M_1-M_2$, the above positivity conditions make this equivalent to
$$
A_1\perp (M_3+D),\quad A_2\perp (M_3-D),\quad A_3\perp (M_3+S),\quad A_4\perp (M_3-S),
$$
where we use the standard inner product $\langle X,Y\rangle={\rm tr}(Y^*X)$ on matrices. Thus with $\widetilde D:=(U^{-1})^*DU^{-1}$,
\begin{equation}\label{orthoD}
UA_1U^*\perp (I_s+\widetilde D),\qquad UA_2U^*\perp (I_s-\widetilde D),
\end{equation}
	and similarly for the other two orthogonality relations involving $\widetilde{S} = (U^{-1})^*SU^{-1}$. Using $-I_s\leq \widetilde D\leq I_s$, the spectral decomposition of $\widetilde D$, and the fact that eigenvectors to different eigenvalues are orthogonal, we see that $UA_1U^*$ and $UA_2U^*$ have orthogonal images, and similarly for $UA_3U^*$ and $UA_4U^*$ with $S$ in place of $D$. This proves~\eqref{orthoims}.

Tracing back this argument, we start with \eqref{orthoims}, construct $\widetilde{D}$ with spectrum in $[-1,+1]$ such that \eqref{orthoD} holds, and similarly for $\widetilde{S}$. This determines $M_1$, $M_2$ and $M_3$ via the above equations, and all desired properties hold by construction.
\end{example}

Before we can prove our main result of this section, we need some more  preliminaries. 

\begin{definition}
$C$ has a {\it universal spectrahedral description} of dimension $r$ if there are $M_1,\ldots,M_d\in{\rm Her}_r(\mathbb C)$ with
$$
\sum_{i=1}^d M_iu_i=I_r,\qquad  C=\mathcal S_1(M_1,\ldots,M_d),
$$
and whenever $N_1,\ldots,N_d\in{\rm Her}_t(\mathbb C)$ with $\sum_i N_iu_i=I_t$, then
$$
\mathcal S_1(M_1,\ldots,M_d)\subseteq \mathcal S_1(N_1,\ldots,N_d)\ \Rightarrow \forall s\geq 1: \mathcal S_s(M_1,\ldots,M_d)\subseteq \mathcal S_s(N_1,\ldots,N_d).
$$
\end{definition} 

This means that the representation detects inclusion of free spectrahedra  already at scalar level. This is closely related to realizations of smallest operator systems:

\begin{proposition}\label{univers}
Let $C\subseteq\mathbb R^d$ be a closed salient cone. Then the following are equivalent:
\begin{itemize}
\item[(i)] The system $C^{\min}$ is finite-dimensional realizable.
\item[(ii)] $C$ admits a universal spectrahedral description.
\end{itemize}
\end{proposition}

\begin{proof}
(i)$\Rightarrow$(ii): Let $M_1,\ldots,M_d$ realize the system. Whenever
$$
C=\mathcal S_1(M_1,\ldots,M_d)\subseteq \mathcal S_1(N_1,\ldots,N_d),
$$
then $\mathcal S_s(M_1,\ldots,M_d)=C_s^{\min}\subseteq \mathcal S_s(N_1,\ldots,N_d)$ for all $s$, since the system is the smallest.

(ii)$\Rightarrow$(i): Let $M_1,\ldots,M_d$ be matrices that form a universal spectrahedral description of $C$. Then $C_s^{\min}\subseteq \mathcal S_s(M_1,\ldots,M_d)$ for all $s\geq 1$. Now assume $A\notin C_t^{\min}$ for some $A\in\mathbb{M}_t(\C)^d$. Then by choosing a separating positive functional and applying Lemma~\ref{effwisep}, there are $N_1,\ldots,N_d$ with $\sum_i N_iu_i=I_t$ and $C^{\min}\subseteq \mathcal S(N_1,\ldots,N_d)$, and such that $A\notin \mathcal S_t(N_1,\ldots,N_d)$. From
$$
\mathcal S_1(M_1,\ldots,M_d)=C=C_1^{\min}\subseteq \mathcal S_1(N_1,\ldots,N_d),
$$
we obtain $\mathcal S(M_1,\ldots,M_d)\subseteq\mathcal S(N_1,\ldots,N_d)$ since the description is universal. Thus $A\notin \mathcal S_t(M_1,\ldots,M_d)$. We have therefore shown $C_s^{\min}=\mathcal S_s(M_1,\ldots,M_d)$ for all $s\geq 1$.
\end{proof}

\begin{lemma}\label{reduce}
Let $H\subseteq \mathbb R^d$ be a subspace that intersects ${\rm int}(C)$, and consider the cone $\tilde C := C\cap H$. Assume that whenever $\tilde C\subseteq \tilde S\subseteq H$ for some spectrahedral cone $\tilde S$, then the matrix pencil defining $\tilde S$ admits an extension to a pencil defining a spectrahedral cone $S\subseteq \mathbb R^d$ containing $C$. If $C^{\min}$ is finite-dimensional realizable, then so is $\tilde{C}^{\min}$.
\end{lemma}

\begin{proof}
Assume $C^{\min}$ is finite-dimensionally realized by $M_1,\ldots, M_d$. Then the restriction of the pencil spaned by $M_1,\ldots,M_d$ to $H$ yields a universal spectrahedral description of $\tilde C$, by the assumed lifting property. In view of Proposition \ref{univers}, $\tilde C^{\min}$ is finite-dimensional realizable. 
\end{proof}

\begin{lemma}\label{face}
Let  $\tilde C$ be a face of $C$. If $C^{\min}$ is finite-dimensional realizable, then so is $\tilde{C}^{\min}$.
\end{lemma}

\begin{proof}
Since $C$ is in particular spectrahedral, $\tilde C$ must be exposed by some $\ell\in C^{\vee}$~\cite{rago}. Then for all matrix levels $s$ and all $A\in\mathbb{M}_s(\C)^d$,
$$
A\in \tilde C_s^{\min}\ \Leftrightarrow\ A\in C_s^{\min} \wedge (\ell\otimes\id)(A)=0.
$$
Thus every finite-dimensional realization of $C^{\min}$ restricts to a finite-dimensional realization of $\tilde C^{\min}$.
\end{proof}

We are now ready to prove the main result of this section:

\begin{theorem}\label{simplex}
For a salient polyhedral cone $C\subseteq\mathbb R^d$, the system $C^{\min}$ is finite-dimensional realizable if and only if $C$ is a simplex. Moreover, $C^{\min}=C^{\max}$ if and only if $C$ is a simplex.
\end{theorem}
\begin{proof}
One direction is easy. Any simplex cone is isomorphic to the positive orthant $C=\R_{\geq 0}^d$. In this case, one easily checks
$$
C_s^{\min}=\left\{ (A_1,\ldots, A_d)\in {\rm Her}_s(\C)^d \Bigm| A_1\geqslant 0,\ldots, A_d\geqslant 0\right\} = C_s^{\max}.
$$
We prove the other direction in $3$ steps.

{\it Step 1}: We first deal with the cone over the square $C={\rm cc}\{v_1,v_2,v_3,v_4\}\subseteq\mathbb R^3$ as in Example~\ref{essbound}, and show that its smallest system is not finite-dimensional realizable. This first nontrivial case is already the hardest. We will use Theorem \ref{finrea} together with our characterization of the essential boundary from Example~\ref{essbound}. Let
$$
\sigma_z=\left(\begin{array}{cc}1 & 0 \\0 & -1\end{array}\right),\qquad \sigma_x=\left(\begin{array}{cc}0 & 1 \\1 & 0\end{array}\right)
$$
be the Pauli matrices. For $\alpha\in(0,\pi/2)$, consider the rank one projections
\begin{align*}
A_1 &= \tfrac{1}{2}\left(I_2 - \cos(\alpha) \sigma_z + \sin(\alpha) \sigma_x\right), \\
A_2 &= \tfrac{1}{2}\left(I_2 + \cos(\alpha) \sigma_z - \sin(\alpha) \sigma_x\right), \\
A_3 &= \tfrac{1}{2}\left(I_2 + \cos(\alpha) \sigma_z + \sin(\alpha) \sigma_x\right), \\
A_4 &= \tfrac{1}{2}\left(I_2 - \cos(\alpha) \sigma_z - \sin(\alpha) \sigma_x\right),
\end{align*}
where the sign pattern is as in~\eqref{squarevertices}, and the associated element
\begin{equation}\label{el}
A := v_1\otimes A_1+v_2\otimes A_2+v_3\otimes A_3+v_4\otimes A_4\in C_2^{\min},
\end{equation}
still parametrized by $\alpha\in(0,\pi/2)$. For $V\in \mathbb{M}_{r,2}(\C)$ with columns $w_1,w_2$, the property
\begin{equation}\label{orth}
\im(VA_1)\perp {\rm im}(VA_2),\quad \im(VA_3)\perp \im(VA_4)
\end{equation}
is equivalent to $\Vert w_1\Vert=\Vert w_2\Vert\ \mbox{ and }\ w_1\perp w_2$, since $A_1$ and $A_2$ are projections onto orthogonal vectors, and likewise for $A_3$ and $A_4$. By taking $V=I_2$, or any other unitary, we conclude $A\in \partial^{\rm ess} C_2^{\min}$ by Example \ref{essbound}. Now let $A^{(1)},\ldots,A^{(r)}$ each be as in (\ref{el}), but for different angles $0<\alpha_1<\ldots<\alpha_r<\tfrac{\pi}{2}$. If these $A^{(i)}$ admit a compression to $\partial^{\rm ess}C_r^{\min}$ as in Theorem \ref{finrea}, then we obtain $V_i\in\mathbb{M}_{r,2}(\C)$ with
\begin{align*}
\im\left(\sum_{i=1}^r V_i A_1^{(i)} V_i^*\right) \:\perp\: \im\left(\sum_{i=1}^r V_i A_2^{(i)} V_i^*\right), \qquad
\im\left(\sum_{i=1}^r V_i A_3^{(i)} V_i^*\right) \:\perp\: \im\left(\sum_{i=1}^r V_i A_4^{(i)} V_i^*\right), \\
\end{align*}
where now the $U$ of \eqref{orthoims} has been absorbed into the $V_i$. Since each summand is positive, these orthogonality relations require the individual summands to have orthogonal images,
\[
\im(V_i A_1^{(i)} V_i^*)\perp \im(V_j A_2^{(j)} V_j^*),\qquad \im(V_i A_3^{(i)} V_i^*)\perp \im(V_j A_4^{(j)} V_j^*),
\]
for all $i,j=1,\ldots,r$. The $A_k^{(i)}$ have rank one, and hence so do the $V_i A_k^{(i)} V_i^*$. An elementary calculation then shows that the $2r$ columns of all the $V_i$'s must be pairwise orthogonal, which is impossible in a space of dimension $r$. Hence $\sum_i V_i A^{(i)} V_i^*$ cannot be in the essential boundary, and Theorem~\ref{finrea} implies that $C^{\min}$ is not $r$-dimensional realizable. So it is not finite-dimensional realizable. This completes Step $1$.

{\it Step 2}: We now generalize to those cones $D$ that fit in between the cone over a square and a circumscribed ellipse, as in Figure~\ref{variousC}, and will then argue that this actually applies to every salient polyhedral cone in $\R^3$. Again  let $C={\rm cc}\{ v_1,\ldots, v_4\}\subseteq \mathbb R^3$ be as in Step $1$. For $\alpha\in (0,\pi/2)$, consider
$$
	M_1 := \sin(\alpha) \sigma_z,\qquad  M_2 := \cos(\alpha) \sigma_x,\qquad M_3 := I_2.
$$
Then  $C\subseteq C(\alpha):=\left\{ (a,b,c)\in \mathbb R^3 \mid aM_1+bM_2+cM_3\geqslant 0\right\}$.  Figure~\ref{variousC} shows sections in the plane defined by $c=1$ of the cones $C(\alpha)$ for various values of $\alpha$.  Again consider $A\in C_2^{\min}$ as in (\ref{el}). The functional $(X,Y,Z)\mapsto {\rm tr}\left(XM_1+YM_2+ZM_3\right)$ even shows that $A \in \partial^{\rm ess} D_2^{\min}$ for any convex cone $D$ with $C\subseteq D\subseteq C(\alpha)$. Now assume that the inclusions
$$
C\subseteq D\subseteq C(\alpha)
$$
hold for {\it infinitely} many values of $\alpha\in (0,\pi/2)$. Then families of $A^{(i)}$ as above (with different values for $\alpha$) are also in $\partial D_2^{\min}$, but cannot  be compressed into  $\partial^{\rm ess}D_r^{\min}$ as in Theorem \ref{finrea}, since this would then also work for $\partial^{\rm ess} C_r^{\min}$. Hence the operator system $D^{\min}$ is not finite-dimensional realizable. 

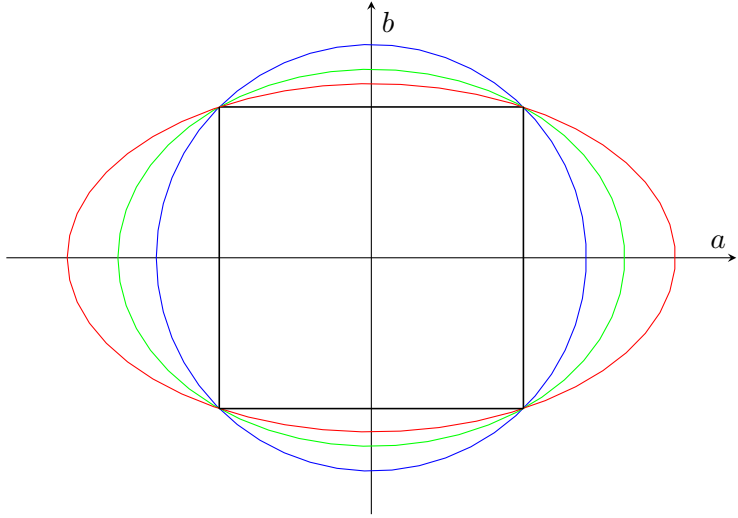
\begin{figure}
\begin{tikzpicture}
\begin{axis}[
    axis x line=middle, 
    axis y line=middle, 
    x=2cm,
    y=2cm,
    xmin=-2.4,
    xmax=2.4,
    ymin=-1.7,
    ymax=1.7,
    ticks=none,
    xlabel={$a$},
    ylabel={$b$},
    domain=-3:3
]
	\addplot[domain=-1:1,semithick] ({x},{1});
	\addplot[domain=-pi:pi,samples=50,blue] ({sqrt(2)*cos(deg(x))},{sqrt(2)*sin(deg(x))});
	\addplot[domain=-pi:pi,samples=50,green] ({5/3*cos(deg(x))},{5/4*sin(deg(x))});
	\addplot[domain=-pi:pi,samples=50,red] ({2*cos(deg(x))},{2/sqrt(3)*sin(deg(x))});
	\addplot[domain=-1:1,semithick] ({x},{-1});
	\addplot[domain=-1:1,semithick] ({1},{x});
	\addplot[domain=-1:1,semithick] ({-1},{x});
\end{axis}
\end{tikzpicture}
\caption{A section of the cone $C$ (square) together with some sections of various $C(\alpha)$.}
\label{variousC}
\end{figure}

To complete Step 2, we show that this applies to \emph{every} salient polyhedral cone $D\subseteq\mathbb{R}^3$. Any quadrilateral in the plane can be transformed by a projective transformation to the square. So in a given planar polytope which is not a simplex, choose vertices $u_1,u_2,w_1,w_2$ that form a quadrilateral, such that both pairs $u_1,u_2$ and $w_1,w_2$ are adjacent vertices. Then transform them to the square, and choose $\alpha'>0$ such that the transformed polytope is contained in $C(\alpha)$ for all $0<\alpha<\alpha'$. This is possible, since the gradient to $\det(aM_1+bM_2+M_3)$ at $(a,b)=(1,1)$ tends to $(1,0)$ for $\alpha\to 0$, and similarly at the other three corners of the square.  This shows that any non-simplex polyhedral cone in $\R^3$ is isomorphic to a cone $D$  with  $C\subseteq D\subseteq C(\alpha)$ for infinitely many values of $\alpha\in (0,\pi/2)$. Its smallest system is thus not finite-dimensional realizable.

{\it Step 3}: We prove the statement in arbitrary dimension $d\geq 4$ by induction on $d$. If $C$ is not a simplex, then either it has a facet that is not a simplex, or a vertex figure that is not a simplex \cite[p.~67]{zieg}. In the first case we apply the contrapositive of Lemma \ref{face}, while in the second case we apply Lemma~\ref{reduce} to a hyperplane defining the vertex figure. The extension required by Lemma \ref{reduce} is possible by taking the conical hull of $\tilde S$ from the vertex (ray). In both cases we reduce to dimension $d-1$.

Finally, the statement about $C^{\min} = C^{\max}$ follows from the previous results.
\end{proof}

\begin{remark}
(i) The argument in Step 2 of the previous proof shows that the smallest system of many non-polyhedral cones is not finite-dimensional realizable either. Any cone in $\R^3$ having a compact section that contains the square and is contained in two different $C(\alpha)$ is an example.

(ii) The results from \cite{ev} provide further evidence that finitely generated operator systems are hardly ever finite-dimensional realizable. 
\end{remark}

\begin{remark}
Every polyhedral cone can be regarded either as the set of positive linear combinations of its finitely many extreme points, or as the set of all points satisfying its finitely many facet inequalities. $C^{\min}$ extends the first picture to matrix levels, since we take matrix positive combinations of points from $C$. On the other hand, $C^{\max}$ generalizes the second picture, since it is defined by the inequalities of $C$. Except for simplices, these two extensions are thus different at matrix level.  
\end{remark}

\begin{example}\label{circle}
There are non-polyhedral cones with a finite-dimensional realizable smallest operator system. One example is the circular cone
$$
C=\left\{ (a,b,c)\in\R^3\mid c\geq  0,\, a^2+b^2\leq c^2\right\}.
$$
It is proven in \cite[Corollary 14.15]{hedi}  and \cite[Theorem 5.4.10]{kelldis} (which relies mostly on \cite[Theorem 7]{choi2}), that the following linear matrix pencil defines the smallest system:
$$
\left(\begin{array}{cc}1 & 0 \\0 & -1\end{array}\right)\otimes x+\left(\begin{array}{cc}0 & 1 \\1 & 0\end{array}\right)\otimes y + \left(\begin{array}{cc}1 & 0 \\0 & 1\end{array}\right)\otimes z.
$$
It is tempting to conjecture that the  following pencil $\mathcal L$ defines the smallest system of the analogous cone over the three-dimensional Euclidean ball in $\R^4$: 
$$
\mathcal L(x,y,w,z) = \left(\begin{array}{cc}1 & 0 \\0 & -1\end{array}\right)\otimes x+\left(\begin{array}{cc}0 & 1 \\1 & 0\end{array}\right)\otimes y + \left(\begin{array}{cc}0 & i \\-i & 0\end{array}\right)\otimes w+\left(\begin{array}{cc}1 & 0 \\0 & 1\end{array}\right)\otimes z ,
$$
which indeed coincides with that cone at the first matrix level. However, this is not true. It is well-known that there are hermitian $4\times 4$-matrices of the block form
$$
X=\left(\begin{array}{cc}A & B \\B^* & C\end{array}\right)
$$
that are positive semidefinite, but cannot be written as $\sum_i P_i\otimes Q_i$ with positive semidefinite matrices $P_i,Q_i$, where all $P_i$ are of size $2$; such matrices are called {\it entangled} in the language of quantum physics. An easy example is the rank one projection
$$
X = \left(\begin{array}{c c|c c}1 & 0 & 0 & 1 \\0 & 0 & 0 & 0 \\\hline 0 & 0 & 0 & 0 \\ 1 & 0 & 0 & 1\end{array}\right).
$$
Now $0\leqslant  2X =\mathcal L\left(A-C,B+B^*,\frac{1}{i}(B-B^*),A+C\right)$. So if the inequality $\mathcal{L}\geqslant 0$ defined the smallest operator system of the cone $C$ over the three-dimensional Euclidean ball, then for every $X$ there would be vectors $v_i\in C$ and positive semidefinite matrices $Q_i$ such that
$$
\left(A-C, B+B^*,\frac{1}{i}(B-B^*),A+C\right)=\sum_i v_i\otimes Q_i.
$$
But then
$$
2X=\sum_i \underbrace{ {} \mathcal L(v_i)}_{\geqslant 0} {} \otimes Q_i,
$$
which contradicts the possibility that $X$ may be entangled.
\end{example}

\section{Inclusion of Spectrahedra}\label{sec_inc}

We explain how our results relate to inclusion testing of spectrahedra.  The inclusion testing problem is the following:

\begin{problem}\label{inprob}
Given $M_1,\ldots, M_d\in {\rm Her}_r(\C)$ and $N_1,\ldots, N_d\in{\rm Her}_t(\C)$ with
$$
\sum_i u_iM_i=I_r,\qquad \sum_i u_iN_i=I_t,
$$
then is it true that
$$
\mathcal S_1(M_1,\ldots,M_d)\subseteq \mathcal S_1(N_1,\ldots, N_d)
$$
holds in $\R^d$?
\end{problem}

Already if $\mathcal S_1(M_1,\ldots,M_d)$ is the cone over a $d$-dimensional cube, this question arises in interesting applications~\cite{bental}. In general, it is a hard algorithmic problem (see \cite{ke3} for an overview and new results). The following strengthening was introduced in~\cite[Section~4.1]{hecp}, generalizing the strengthening of~\cite[Eq.~(7)]{bental} for the matrix cube problem:

\begin{problem}[{\cite{hecp}}]
\label{relprob}
Given $M_1,\ldots, M_d\in {\rm Her}_r(\C)$ and $N_1,\ldots, N_d\in{\rm Her}_t(\C)$ with
$$
\sum_i u_iM_i=I_r,\qquad \sum_i u_iN_i=I_t,
$$
do there exist $V_j\in \mathbb M_{r,t}(\C)$ such that
$$
\sum_j V_j^*M_iV_j=N_i
$$
for all $i$? 
\end{problem}

A positive answer to an instance of Problem \ref{relprob} implies a positive answer to the corresponding instance of Problem \ref{inprob}. Furthermore, Problem \ref{relprob} can be formulated as a semidefinite feasibility problem, and is thus algorithmically tractable.
However, a positive answer to Problem \ref{inprob} does not necessarily imply a positive answer to Problem \ref{relprob}. 
The main result of \cite{hecp}  says that Problem \ref{relprob} is equivalent to $\mathcal S_s(M_1,\ldots,M_d) \subseteq \mathcal S_s(N_1,\ldots,N_d)$ for {\it all} $s\geq 1$, i.e.~to inclusion of the free spectrahedra. This result mostly relies on Choi's characterization of completely positive maps between matrix algebras \cite{choi2}. Since the inclusion  $\mathcal S_1(M_1,\ldots,M_d)\subseteq \mathcal S_1(N_1,\ldots, N_d)$ does not imply the higher inclusions  $\mathcal S_s(M_1,\ldots,M_d) \subseteq \mathcal S_s(N_1,\ldots,N_d)$ in general, Problem \ref{relprob} is a proper strengthening of Problem \ref{inprob}.
There exist quantitative measures for tightness of this strengthening \cite{bental, dav2,hedi, ke2}, which we will explain in more detail below. The first reformulation of our previous result is the following:

\begin{corollary}\label{incmain}
Assume $C=\mathcal S_1(M_1,\ldots,M_d)\subseteq \R^d$ is a salient polyhedral cone. Then a positive answer to Problem \ref{inprob} implies a positive answer to Problem \ref{relprob} for all choices of $N_1,\ldots,N_d\in{\rm Her}_t(\C)$ if and only if $C$ is a simplex.
\end{corollary}

Of course, even when $C$ is not a simplex, the strengthening may still give the correct answer for a \emph{particular} choice of $N_1,\ldots,N_d$.

\begin{proof}
Problems \ref{inprob} and \ref{relprob} are equivalent for all $N_1,\ldots,N_d$ if and only if the $M_i$ provide a universal spectrahedral description of $C$. So the result follows from Proposition \ref{univers} and Theorem \ref{simplex}.
\end{proof}

\begin{remark}
(i) Although the tightness of the strengthening for simplex cones is easy to prove, it seems like it has not been observed in the literature so far. We will use it below to easily derive error bounds for the non-tight case.

(ii) Corollary \ref{incmain} holds for {\it any description} of $C$ by matrices $M_i$.  So far, only fixed descriptions have been used to deduce error bounds and non-tightness results in~\cite{bental,dav2,hedi,ke2}. It was not clear a priori whether choosing a better spectrahedral description of the cones could result in tightness.  We now know that this is impossible for non-simplex polyhedral cones.

(iii) Finally, we are not aware of any implications of our results on the tightness of the higher levels of the semidefinite hierarchy of~\cite{ke2}.
\end{remark}

Using our approach, parts of~\cite[Theorem 4.8]{ke2} become easy to prove. The result implies that for inclusion of spectrahedra in polyhedra, the strengthening is always tight.

\begin{proposition}
\label{incinpoly}
If $N_1,\ldots,N_d$ commute and $C \subseteq \mathcal S_1(N_1,\ldots, N_d)$, then
$$
C^{\max}\subseteq \mathcal S(N_1,\ldots,N_d).
$$
\end{proposition}

\begin{proof}
$P=\mathcal{S}_1(N_1,\ldots, N_d)$ is polyhedral, and thus $P^{\max}=\mathcal{S}(N_1,\ldots, N_d)$ by the easy direction in the proof of Theorem \ref{maxreal}. The claim now follows from $C^{\max}\subseteq P^{\max}$.
\end{proof}

\newcommand{\scaledcone}[2]{#2\!\uparrow\! #1} 

As explained above, a positive answer to the strengthened Problem~\ref{relprob} implies a positive answer to the original Problem~\ref{inprob}. In the situation of Corollary~\ref{incmain} or Proposition~\ref{incinpoly}, also a negative answer to Problem~\ref{relprob} implies the same for Problem~\ref{inprob}, although such an inference is not valid in general. One way to approach this issue is to use an entire hierarchy of semidefinite programs that converge to Problem~\ref{inprob}~\cite{ke2}. Another one is to modify the formulation of Problem~\ref{relprob} so as to make the implication work; concretely, we can replace the $M_i$ in Problem~\ref{relprob} by a ``scaled down'' version which is small enough for the implication of negative answers to be valid.

For a salient convex cone $C$, the definition of scaling is as follows. Choose an arbitrary hyperplane $H$ that intersects $C$ only at the origin. For a scaling factor $\nu > 0$, the scaled cone $\scaledcone{C}{\nu}$ is constructed by taking the intersection of $C$ with the affine hyperplane $u+H$, scaling this intersection by the factor $\nu$ from the point $u$, and taking the conical hull with the origin again, resulting in $\scaledcone{C}{\nu}$. In general, $\scaledcone{C}{\nu}$ depends on the choices of $H$ and $u$.

While $C^{\min}\subseteq C^{\max}$ holds trivially, we can now ask by how much $C^{\min}$ and $C^{\max}$ differ, namely by investigating how small the scaling factor $\nu>0$ needs to be in order for the reverse inclusion to hold,
\begin{equation}
\label{coneinc}
(\scaledcone{C}{\nu})^{\max}\subseteq C^{\min}.
\end{equation}
We now derive some results on this and then get back to the relation between Problems~\ref{inprob} and~\ref{relprob}.

\begin{proposition}\label{scale}
Let $C$ be a closed salient cone. Then for any choice of $H$ and $u$ there is some $\nu>0$ such that~\eqref{coneinc} holds.
\end{proposition}

\begin{proof}
After choosing $H$ and $u$, choose $\nu>0$ and a simplex cone $S$ with $$\scaledcone{C}{\nu} \subseteq S\subseteq C.$$ We then have
\[
(\scaledcone{C}{\nu})^{\max} \subseteq S^{\max}=S^{\min} \subseteq C^{\min}. \qedhere
\]
\end{proof}

For any inclusion of cones $C\subseteq D$, we thus also have $(\scaledcone{C}{\nu})^{\max}\subseteq D^{\min}$. By suitable choice of $u$, we can also find a uniform bound on $\nu$ that only depends on the dimension:

\begin{theorem}
Let $C\subseteq \R^d$ be a closed salient cone. Then for any choice of $H$, there is an order unit $u\in C$ such that the inclusion $(\scaledcone{C}{\nu})^{\max}\subseteq C^{\min}$ holds with $\nu=1/(d+1)$.
\end{theorem}

\begin{proof}
In the proof of Proposition \ref{scale}, we apply the main theorem of \cite{las}: whenever one inscribes into a convex body in $\R^{d-1}$ a simplex of maximal volume, then scaling the body with ratio $1/(d+1)$ from the barycenter of the simplex will make it contained in the simplex.
\end{proof}

We can also recover the factor of inverse dimension from \cite{dav2} in the presence of symmetry. Since we talk about cones as opposed to compact convex bodies, this dimension is our $d-1$:

\begin{theorem}\label{symm}
Let $C\subseteq \R^d$ be a closed salient cone, and assume $C\cap (u+H)$ is symmetric with respect to $u$. Then the inclusion~\eqref{coneinc} holds with $\nu=1/(d-1)$.
\end{theorem}

\begin{proof}
As in the previous proof, one can use the simplex of maximal volume contained in the centrally symmetric convex body $C\cap (u+H)$. Gr\"unbaum showed that the scaling factor can then be taken equal to the dimension~~\cite[p.~259]{gbaum}. But since his method would not necessarily yield the center of symmetry $u$ as the center of scaling, we argue slightly differently.

For notational simplicity, we assume $H=\R^{d-1}\times\{0\}\cong \R^{d-1}$ and $u=(0,\dots,0,1)$, which we take to be the origin of $u+H$ as identified with $\R^{d-1}$. Set $\tilde C:=C\cap (u+H)\subseteq\R^{d-1}$ and let $S\subseteq \tilde C$ be a simplex of maximal volume. Let $b$ denote the barycenter of $S$. If $F$ is a face of $S$ and $v$ its opposite vertex, then the translate of $F$ through $v$, which is $F + \tfrac{d}{d-1}(v-b)$, is a hyperplane supporting $\tilde C$, since otherwise we could increase the volume of the simplex. The simplex defined by all these translates is $-(d-1)(S-b)+b$, and therefore $\tilde C\subseteq -(d-1)(S-b) + b$. Symmetry of $\tilde C$ then implies
$$
\frac{1}{d-1}\tilde C +\frac{d}{d-1}b\subseteq S.
$$
Now assume $(A_1,\ldots, A_d)\in C_s^{\max}$. Then
$$
\left(\frac{1}{d-1}A_1+\frac{d}{d-1}b_1A_d,\ldots, \frac{1}{d-1}A_{d-1}+\frac{d}{d-1}b_{d-1}A_d,A_d\right)\in C_s^{\min},
$$
by the argument used for Proposition \ref{scale}. By symmetry we get the same result with $-b$ instead of $b$, and after adding and dividing by $2$, we arrive at the desired conclusion,
\[
\left(\frac{1}{d-1}A_1,\ldots, \frac{1}{d-1}A_{d-1},A_d\right)\in C_s^{\min}.\qedhere
\]
\end{proof}

\begin{remark}
Now given an instance of Problem~\ref{inprob}, choose $\nu > 0$ such that~\eqref{coneinc} holds for $C = \mathcal{S}_1(M_1,\ldots,M_d)$. For convenience of notation, chose coordinates such that $u = (0,\ldots,0,1)$ and $H = \mathbb{R}^{d-1}\times \{0\}$ as in the previous proof. Then
\[
\scaledcone{C}{\nu} = \mathcal{S}_1(\nu^{-1} M_1,\ldots,\nu^{-1}M_{d-1},M_d).
\]
Now property~\eqref{inprob} guarantees that if Problem~\ref{relprob} with
\[
M'_i := \begin{cases} \nu^{-1} M_i & \textrm{for } i < d, \\ M_d & \textrm{for } i = d. \end{cases}
\]
in place of the $M_i$ has a negative solution, then so does the original Problem~\ref{inprob}. So with this modification, Problem~\ref{relprob} is a relaxation rather than a strengthening of Problem~\ref{inprob}. Intuitively speaking, the closer the scaling factor $\nu$ is to $1$, the smaller the gap between Problem~\ref{inprob} and Problem~\ref{relprob} will be.

In other applications, such as the matrix cube problem~\cite{bental}, one is directly interested in the largest $\nu$ for which an inclusion of the form $\scaledcone{\mathcal{S}_1(M_1,\ldots,M_d)}{\nu}\subseteq \mathcal{S}_1(N_1,\ldots,N_d)$ holds. In this situation, semidefinite strengthening as in Problem~\ref{relprob} provides a lower bound on the optimal $\nu$. In the case of the matrix cube problem, Theorem~\ref{symm} applies, and we may conclude that the strengthening of Problem~\ref{relprob} differs by a factor of at most $d-1$ from the actual optimal value. This bound is neither dominated by nor dominating over the error bound of Ben-Tal and Nemirovski~\cite{bental}.
\end{remark}

\section*{Acknowledgments}
We thank Kai Kellner and Markus Schweighofer for interesting discussions, and Bill Helton, Igor Klep, Scott McCullough and two anonymous referees for helpful feedback on ealier versions.

The second author was supported by Grant No.\ P 29496-N35 of the Austrian Science Fund (FWF). 
The third author was supported by ERC Starting Grant No.\ 277728 and ERC Consolidator Grant No.\ 681207.

\end{document}